\documentclass[12pt,oneside]{article}
\usepackage{amsmath,amssymb,amsfonts,amsthm}
\usepackage{color}
\usepackage[all]{xy}
\usepackage{graphicx}
\usepackage{color}
\usepackage{makeidx}
\usepackage{multirow}
\usepackage{rotating}
\usepackage{pdfpages}
\usepackage{pdflscape}

\numberwithin{equation}{section} 
\numberwithin{figure}{section} 

\theoremstyle{plain}
\newtheorem*{theorem*}{Theorem}
\newtheorem{theorem}{Theorem}[section]

\newtheorem{proposition}[theorem]{Proposition}
\newtheorem{corollary}[theorem]{Corollary}
\newtheorem{remark}[theorem]{Remark}

\newtheorem*{acknowledgement*}{Acknowledgement}

\numberwithin{equation}{section}

\newcommand\overcirc[1]{\raisebox{10pt}{\tiny{$\circ$}}{\kern-7.5pt}\mbox{$#1$}}
\newcommand\undersym[2]{\raisebox{-6pt}{$#2$}{\kern-5pt}\mbox{$#1$}}
\newcommand\overdiamond[1]{\raisebox{10pt}{\small$\star$}{\kern-7.5pt}\mbox{$#1$}}
\newcommand\overast[1]{\raisebox{10pt}{\small$\ast$}{\kern-7.5pt}\mbox{$#1$}}
\newcommand\overlind[1]{\raisebox{10pt}{\small$\overline{{\hspace{2pt}}\star}$}{\kern-7.5pt}\mbox{$#1$}}
\newcommand\overlinc[1]{\raisebox{10pt}{\small$\overline{{\hspace{2pt}}\circ}$}{\kern-7.5pt}\mbox{$#1$}}
\newcommand\overlina[1]{\raisebox{10pt}{\small$\overline{{\hspace{1pt}}\ast}$}{\kern-7.5pt}\mbox{$#1$}}

\begin{document}

\title{ A note on \lq\lq On the classification of Landsberg spherically symmetric Finsler metrics\rq\rq}
\author{S. G. Elgendi}
\date{}

\maketitle
\vspace{-1 cm}

\begin{center}
{Departmernt of Mathematics, Faculty of Science,\\
Islamic University of Madinah, Madinah, Saudi Arabia}
\vspace{-8pt}
\end{center}

\begin{center}
{Department of Mathematics, Faculty of Science,\\
Benha University, Benha, Egypt}
\vspace{-8pt}
\end{center}

\begin{center}
salah.ali@fsci.bu.edu.eg, salahelgendi@yahoo.com
\end{center}

\vspace{0.3cm}
\begin{abstract} 
In this  paper,    we prove that all spherically symmetric Landsberg surfaces are Berwaldian. We modify  the classification of   spherically symmetric Finsler  metrics, done by the author
in [S. G. Elgendi, On the classification of Landsberg spherically symmetric Finsler metrics,
Int. J. Geom. Methods Mod. Phys. 18 (2021)],  of Berwald type of dimension $n\geq 3$. Precisely, we show that all Berwald spherically symmetric  metrics of dimension $n\geq 3$ are Riemannian or given by a certain formula. As a simple class of Berwaldian metrics, we prove that all spherically symmetric metrics in which  the function $\phi$ is homogeneous of degree $-1$ in $r$ and $s$ are Berwaldian.   
\end{abstract}




\section{Introduction}
~\par

In Finsler geometry, the  existence of a regular non-Berwladian Landsberg Finsler metric   is still an open problem. In the two-dimensional (2D) case,   that   problem seems  more complicated.  Some non-regular  Landsberg Finsler metrics which are not Berwladian are known in higher dimensions (cf. \cite{Asanov,Elgendi-solutions,Shen_example}).  But in dimension two,  to the best of our knowledge, no concrete examples are given. There is a class of examples of non-Berwaldian Landsberg spherically symmetric surfaces  obtained by L. Zhou \cite{Zhou}.  But in a joint paper of the author (cf. \cite{Elgendi-Youssef}),  it was proven that this class is, in   fact, Berwaldian.

\medskip

  In \cite{Elgendi-SSM}, we have classified all Landsberg spherically symmetric Finsler metrics of dimension $n\geq 3$.  Precisely, we prove that all Landsberg spherically symmetric   metrics of dimension $n\geq 3$ are Riemannian or its geodesic spray is given by a certain formula. In this paper, we complete this classification by showing that all 2D Landsberg spherically symmetric Finsler metrics (regular or non-regular)  are Berwladian.

\medskip 

Also, in \cite{Elgendi-SSM}, we have  proven that all Berwaldian spherically symmetric   metrics of dimension $n\geq 3$ are Riemannian. But we discovered a missing case in the proof. In this paper, we modify this result and prove that all Berwaldian spherically symmetric   metrics of dimension $n\geq 3$ are Riemannian or the function $\phi$ is given by
  $$
\phi= s\ \psi\left(\frac{s^2}{g(r)+s^2 \int 4 r c_0(r) g(r) d r}\right) e^{-\int\left(\frac{2}{r}-2 r^3 c_0(r)\right) dr} 
$$
where $c_0(r)$ is a smooth function and $g(r)=e^{\int\left(\frac{2}{r}-4 r^3 c_0(r)\right) dr}$.

\medskip

At the end of this paper we provide a table classifying all spherically symmetric Finsler metrics of Landsberg and Berwald types. 
\section{Spherically symmetric metrics}

~\par Throughout,  we use the notations and terminology of \cite{Elgendi-SSM}.
  A spherically symmetric Finsler metric $F$ on $\mathbb{B}^n(r_0)\subset\mathbb{R}^n$ is defined by
$$F(x,y)=|y|\ \phi\left(|x|,\frac{\langle x, y \rangle}{|y|}\right),$$

where   $\phi:[0,r_0)\times\mathbb{R}^n\to \mathbb{R}$, $(x,y)\in T\mathbb{B}^n(r_0)\backslash \{0\}$, and $|\cdot|$ and $\langle \cdot , \cdot \rangle$ are the standard Euclidean norm and inner product on $\mathbb{R}^n$. Or simply,  $F=u\ \phi(r,s)$ where $r=|x|$, $u=|y|$ and $s=\frac{\langle x, y \rangle}{|y|}$.

 \medskip

The spherically symmetric metrics are a special  general $(\alpha,\beta)$-metrics.  Therefore, a spherically symmetric metric  $F=u\phi(r,s)$ on $\mathbb{B}^n(r_0)$ is regular if and only if $\phi$ is positive,  $C^\infty$ function such that
\begin{equation*}
\label{Regular_condition}
\phi-s\phi_s>0,\quad \phi-s\phi_s+(r^2-s^2)\phi_{ss}>0.
\end{equation*}
Moreover, in case of $n=2$, the regularity condition is  $$\phi-s\phi_s+(r^2-s^2)\phi_{ss}>0$$
    for all $|s|\leq r<r_0$. The subscript $s$ (resp. $r$) refers to the derivative with respect to $s$ (rep. $r$). 
The spherically symmetric Finsler metrics are studied in many papers for more details, we refer   to, \cite{Guo-Mo,Zhou_Mo,Zhou}.

\medskip

Since the components of the metric tensor associated with  the Euclidean norm are just the Kronecker  delta $\delta_{ij}$, then we  lower the indices of $y^i$ and $x^i$ as follows
$$y_i:=\delta_{ih} y^h, \quad x_i:=\delta_{ih} x^h.$$
It should be noted that $y_i$ and $x_i$ are the same as $y^i$ and $x^i$ respectively. So we confirm that $y_i\neq F\frac{\partial F}{\partial y^i}$ rather $y_i= u\frac{\partial u}{\partial y^i}$. Moreover, we have the following properties
$$y^iy_i=u^2, \quad x^ix_i=r^2, \quad y^ix_i=x^iy_i=\langle x,y \rangle.$$

 By making use of the above notations and properties, we are able to  keep the indices  consistent.   Many articles in the literature study the geometric objects associated to  a spherically symmetric Finsler metric   with some kind of inconsitency  with the indices.  For example, in \cite{Zhou},   a tensorial equation has an up index in one side and in the other side the same index is lower index. We tried to  solve this problem in \cite{Elgendi-SSM} but with a bit long formulae. In this paper, we fix it with a simple and natural way.
 
 \medskip

The components $g_{ij}$ of the metric tensor of the spherically symmetric metric $F=u \phi(r,s)$ are given by  
\begin{equation}
\label{Eq:g^ij}
g_{ij}=\sigma_0\ \delta_{ij}+\sigma_1\  x_ix_j+\frac{\sigma_2}{u} (x_iy_j+x_jy_i)+\frac{\sigma_3}{u^2}y_iy_j,
\end{equation}
where  $$\sigma_0=\phi(\phi-s\phi_s),\quad \sigma_1= \phi_s^2+\phi\phi_{ss},\quad \sigma_2= (\phi-s\phi_s)\phi_s-s\phi\phi_{ss},  \quad \sigma_3= s^2\phi\phi_{ss}-s(\phi-s\phi_s)\phi_s.$$

The following geometric objects can be found in \cite{Guo-Mo,Zhou_Mo,Zhou}. 
The components $g^{ij}$ of the inverse metric tensor are given  as follows
\begin{align}
\label{Eq:g^ij}
g^{ij}=&\rho_0\delta^{ij}+ \frac{\rho_1}{u^2}y^iy^j+ \frac{\rho_2}{u} (x^iy^j+x^jy^i)+\rho_3x^ix^j,
\end{align}
where
$$
 \rho_0=\frac{1}{\phi(\phi-s\phi_s)}, \quad \rho_1=\frac{(s\phi+(r^2-s^2)\phi_s)(\phi\phi_s-s\phi_s^2-s\phi\phi_{ss})}{\phi^3(\phi-s\phi_s)(\phi-s\phi_s+(r^2-s^2)\phi_{ss})},$$
 $$\rho_2=-\frac{\phi\phi_s-s\phi_s^2-s\phi\phi_{ss}}{\phi^2(\phi-s\phi_s)(\phi-s\phi_s+(r^2-s^2)\phi_{ss})},\quad \rho_3=-\frac{\phi_{ss}}{\phi(\phi-s\phi_s)(\phi-s\phi_s+(r^2-s^2)\phi_{ss})}.$$
  The coefficients  $G^i$ of the geodesic spray  of $F$ are  given by
\begin{equation}\label{G}
  G^i=uPy^i+u^2 Qx^i,
\end{equation}
where the functions $P$ and $Q$ are defined by
\begin{equation}\label{P,Q}
 Q:=\frac{1}{2 r}\frac{ -\phi_r+s\phi_{rs}+r\phi_{ss}}{\phi-s\phi_s+(r^2-s^2)\phi_{ss}}, \quad P:=-\frac{Q}{\phi}(s\phi +(r^2-s^2)\phi_s)+\frac{1}{2r\phi}(s\phi_r+r\phi_s).
\end{equation}

The components $G^i_{jk\ell}$ of the Berwald curvature are defined by
$$G^i_{jk\ell}=\frac{\partial^2 }{\partial y^\ell \partial y^k\partial y^j} G^i.$$
For a spherically symmetric Finsler metric $F=u\ \phi(r,s)$,  the components $G^i_{jk\ell}$ are calculated as follows 
\begingroup
\allowdisplaybreaks
\begin{align}
\label{Eq:G^i_{jkl}}
   \nonumber     G^i_{jk\ell}= & \frac{P_{ss}}{u}(\delta^i_j x_k x_\ell+\delta^i_\ell x_j x_k+\delta^i_k x_j x_\ell) +\frac{1}{u}(P-sP_s)(\delta^i_j \delta_{k\ell}+\delta^i_k \delta_{j\ell}+\delta^i_\ell \delta_{jk})\\
     \nonumber    &-\frac{sP_{ss}}{u^2}\left(\delta^i_j\left(x_k y_\ell+x_\ell y_k\right)+\delta^i_k\left(x_j y_\ell+x_\ell y_j\right)+\delta^i_\ell\left(x_j y_k+x_k y_j\right)\right) \\
     \nonumber    &-\frac{sP_{ss}}{u^2}\left(\delta_{j k} x_\ell+\delta_{j \ell} x_k+\delta_{k \ell} x_j\right) y^i
        +\frac{1}{u}(Q_s-sQ_{ss})\left(\delta_{j k} x_\ell+\delta_{j \ell} x_k+\delta_{k \ell} x_j\right)x^i \\
    \nonumber     &+\frac{1}{u^3}(s^2P_{ss}+sP_s-P)\left(\left(\delta^i_j y_k y_\ell+\delta^i_k y_j y_\ell+\delta^i_\ell y_j y_k\right)+\left(\delta_{j k} y_\ell+\delta_{j \ell} y_k+\delta_{k \ell} y_j\right) y^i\right)\\
        &+\frac{1}{u^5}(3P-s^3P_{sss}-6s^2P_{ss}-3sP_s) y_jy_ky_\ell y^i+\frac{P_{sss}}{u^2} x_jx_kx_\ell y^i\\
   \nonumber      &+\frac{1}{u^4}(s^2P_{sss}+3sP_{ss})\left(y_j y_k x_\ell+y_j y_\ell x_k+y_k y_\ell x_j\right)y^i\\
   \nonumber      &-\frac{1}{u^3}(P_{ss}+sP_{sss})\left(y_j x_k x_\ell+y_k x_j x_\ell+y_\ell x_j x_k\right)y^i+\frac{Q_{sss}}{u}x_jx_kx_\ell x^i\\
    \nonumber     &+\frac{1}{u^3}(s^2Q_{sss}+sQ_{ss}-Q_s)\left(x_j y_k y_\ell+x_k y_j y_\ell+x_\ell y_j y_k\right)x^i\\
    \nonumber  &
        -\frac{sQ_{sss}}{u^2}(x_j x_\ell y_k+x_j x_k y_\ell+x_k x_\ell y_j)x^i
        +\frac{1}{u^4}  (3sQ_s-3s^2Q_{ss}-s^3Q_{sss})y_j y_k y_\ell x^i\\
   \nonumber      &+\frac{1}{u^2}(s^2Q_{ss}-sQ_s)(\delta_{k\ell }y_j+\delta_{j\ell }y_k+\delta_{kj }y_\ell )x^i.
\end{align}
\endgroup

The components $E_{ij}:=G^h_{ijh}$ of the mean Berwald curvature is given by  
\begin{equation}
\label{Mean_curv.}
 \begin{split}
E_{ij} =&\frac{1}{u}((n+1)(P-sP_s)+(r^2-s^2)(Q_s-sQ_{ss}))\delta_{ij}+ \frac{1}{u^3}((n+1)(s^2P_{ss}+sP_s-P)\\
&+r^2(s^2Q_{sss}+sQ_{ss}-Q_s)+3s^2Q_s-3s^3Q_{ss}-s^4Q_{sss})  y_i y_j\\
&+\frac{1}{u}((n+1)P_{ss}+2(Q_s-sQ_{ss})+(r^2-s^2)Q_{sss})x_ix_j\\
&-\frac{s}{u^2}((n+1)P_{ss}+2(Q_s-sQ_{ss})+(r^2-s^2)Q_{sss})(x_i y_j+x_j y_i)  .
\end{split}
\end{equation}
One can rewrite  $E_{ij}$ as follows 
\begin{equation} 
\label{E_H}
E_{ij}=\frac{H}{u}\delta_{ij}-\frac{1}{u^3}(sH_s+H) y_iy_j+\frac{H_s}{su^2}(s(x_i y_j+x_j y_i)-u x_ix_j),
\end{equation}
where 
 $$H:=(n+1)(P-sP_s)+(r^2-s^2)(Q_s-sQ_{ss}).$$
 
 In what follow, we recall some results of \cite{Elgendi-SSM}.
 
 \begin{proposition}\cite{Elgendi-SSM}
Let $P(r,s)$ and $Q(r,s)$ be given, then the Finsler function $F=u\phi(r,s)$ whose geodesic spray given by $P$ and $Q$ is determined by the function $\phi$ provided that $\phi$ satisfies the following   compatibility conditions:
\begin{equation}
\label{Comp_C_C_2}
    \begin{split}
       C_1:= &(1+sP-(r^2-s^2)(2Q-sQ_s))\phi_s+(s P_s-2P-s(2Q-sQ_s))\phi =0,   \\
        C_2:=& \frac{1}{r}\phi_r-(P+Q_s(r^2-s^2))\phi_s-(P_s+sQ_s) \phi =0.
    \end{split}
\end{equation}
\end{proposition}

\begin{theorem}\cite{Elgendi-SSM}\label{Theorem_A}
A Landsberg spherically symmetric Finsler metric of dimension $n\geq 3$ is either Riemannian or the geodesic  spray is determined  by the functions 
 \begin{equation}
\label{Zhou_P&Q}
P=c_1 s+\frac{c_2}{r^2}\sqrt{r^2-s^2}  , \quad Q=\frac{1}{2}c_0 s^2-\frac{c_2 s}{r^4}\sqrt{r^2-s^2}+c_3,
\end{equation}
where $c_0$, $c_1$, $c_2$, $c_3$ are arbitrary functions of $r$.  
\end{theorem}

 \begin{theorem} \cite{Elgendi-SSM}
  \label{First_surface_result}
A spherically symmetric Finsler surface is Berwaldian if and only if 
$$P=b_1 s+ \frac{b_2}{ \sqrt{r^2-s^2}} +\frac{b_3 (r^2-2s^2)}{\sqrt{r^2-s^2}} ,$$
$$ Q= b_0 s^2+\frac{1}{2} b_1+ \frac{b_2 s (r^2-2 s^2)}{r^4 \sqrt{r^2-s^2}} -\frac{b_3 s (3r^2-2s^2)}{r^2\sqrt{r^2-s^2}}-\frac{a  }{r^2} s \sqrt{r^2-s^2},$$
where $a$, $b_0$, $b_1$, $b_2$, $b_3$ are arbitrary functions of $r$ and to be chosen such  that the  compatibility conditions are satisfied.
\end{theorem}

\begin{proposition}\cite{Elgendi-SSM}\label{2D_L_ijk=0}
A spherically symmetric surface is Landsbergian  if and only if $ (r^2-s^2)L_1+3L_2=0$, that is,
\begin{equation}
\label{Main_Lands_surface_cond}
\begin{split}
(r^2-s^2)L_1+3L_2=&((r^2-s^2)((r^2-s^2)Q_{sss}+3(Q_s-sQ_{ss})+3P_{ss})+3(P-sP_s))\phi_s\\
&+((r^2-s^2)(sQ_{sss}+P_{sss})+3s(Q_s-sQ_{ss})-3sP_{ss})\phi\\
=&\frac{1}{s}(sH-(r^2-s^2)H_s) \phi_s+((r^2-s^2)K_s-3sK)\phi=0,
\end{split}
\end{equation}
where
$  K:=P_{ss}-Q_s+sQ_{ss}.$
 \end{proposition}

  We end this section by proving the following result.
  \begin{proposition}\label{Berwald_surface}
  A spherically symmetric Finsler surface $F=u\phi$ is Berwaldian if and only if 
  $$sH-(r^2-s^2)H_s=0.$$
  \end{proposition}
  \begin{proof}
  Let $F=u\phi$  be a Berwaldian spherically symmetric Finsler surface, then the Berwald curvature  vanishes. Then, we conclude that    the mean curvature $E_{ij}$ vanishes as well. Now, by \cite[Proposition 6.1]{Elgendi-SSM}, we obtain that $sH-(r^2-s^2)H_s=0.$
  
  Conversely, assume that  $sH-(r^2-s^2)H_s=0$. Then, by the proof of Theorem 6.3 in \cite{Elgendi-SSM}, the functions $P$ and $Q$ are given by Theorem \ref{First_surface_result}. Consequently, $F$ is Berwaldain. 
  \end{proof}
\section{Berwald case}

The following theorem is a modified version of \cite[Theorem 5.4]{Elgendi-SSM}.

\begin{theorem}\label{Th_Riemannian}
All   Berwaldian spherically symmetric metrics of dimension $n\geq3$ are Riemannian or the function $\phi$ is given by
$$
\phi= s\ \psi\left(\frac{s^2}{g(r)+s^2 \int 4 r c_0(r) g(r) d r}\right) e^{-\int\left(\frac{2}{r}-2 r^3 c_0(r)\right) dr} 
$$
where $c_0(r)$ is a smooth function and $g(r)=e^{\int\left(\frac{2}{r}-4 r^3 c_0(r)\right) dr}$.
\end{theorem}
\begin{proof}
Let $F$ be a Berwald spherically symmetric of  dimension $n\geq3$.  Since every Berwald metric is Landsbergian, then the geodesic spray  of $F$ is determined by \eqref{Zhou_P&Q};
$$P=c_1 s+\frac{c_2}{r^2}\sqrt{r^2-s^2}  , \quad Q=\frac{1}{2}c_0 s^2-\frac{c_2 s}{r^4}\sqrt{r^2-s^2}+c_3.$$
 Calculating the quantities
 $$P-sP_s=\frac{c_2}{\sqrt{r^2-s^2}}, \quad Q_s-sQ_{ss}=-\frac{c_2}{(r^2-s^2)^{3/2}},$$
  $$ P_{ss}=-\frac{c_2}{(r^2-s^2)^{3/2}}, \quad  Q_{sss}=\frac{3c_2}{(r^2-s^2)^{5/2}}.$$
 Since $F$ is Berwaldian,  then the mean curvature $E_{ij}=0$. Now, plugging the above quantities into the equation $E_{ij}=0$ implies
$$\frac{n c_2}{u\sqrt{r^2-s^2}}\left(  \delta_{ij}-\frac{r^2}{u^2(r^2-s^2)} y_iy_j -\frac{1}{r^2-s^2} x_ix_j +\frac{s}{u(r^2-s^2)}(x_iy_j+x_jy_i)\right)=0.$$
Contracting the above equation by $\delta^{ij}$ and using the properties $\delta^{ij}y_iy_j=u^2$, $\delta^{ij}x_ix_j=r^2$ and $\delta^{ij}x_iy_j=\langle x,y\rangle$, we have
$$ n (n-2)c_2 \sqrt{r^2-s^2}=0. $$
Since $n\geq3$ and the above equation holds for all $r$ and $s$, we must have $c_2=0$. Thus, we have
\begin{equation}
\label{Eq:Berwald-P-Q}
P=c_1 s  , \quad Q=\frac{1}{2}c_0 s^2+c_3.
\end{equation}
 Plugging   the above formulae of $P$ and $Q$ into the compatibility conditions \eqref{Comp_C_C_2}, we get

\begin{equation*}
    \begin{split}
       C_1 = &(1-2c_3r^2+(c_1+2c_3)s^2)\phi_s-s( c_1+2c_3)\phi =0,   \\
        C_2 =& \frac{1}{r}\phi_r-(c_1s+c_0 s(r^2-s^2))\phi_s-(c_1+c_0s^2) \phi =0.
    \end{split}
\end{equation*}
Now, we have two cases; the first case is $1-2c_3r^2+(c_1+2c_3)s^2\neq 0$, then solving the above two equations algebraically  implies
\begin{equation}
\label{Zhou_phi_1}
    \begin{split}
        \frac{\phi_s}{\phi}= & \frac{(c_1+2c_3)s }{1+(c_1+2c_3)s^2-2c_3r^2 },   \\
         \frac{\phi_r}{\phi}=&   \frac{r(c_0(1+c_1r^2)s^2+2c_1(c_1+2c_3)s^2+c_1(1-2c_3r^2))}{1+(c_1+2c_3)s^2-2c_3r^2}.
    \end{split}
\end{equation}
 Integrating $ \frac{\phi_s}{\phi}$ with respect to $s$ yields 
$$\phi=a(r) \sqrt{(c_1+2c_3) s^2-2c_3 r^2+1},$$
where $a(r)$ is to be chosen such that both formulae of  \eqref{Zhou_phi_1} are satisfied, that is, calculating $\frac{\phi_r}{\phi}$ and equaling it with the second formula of \eqref{Zhou_phi_1} we obtain $\phi$.   Consequently, the metric $$F=u\phi= a(r)    \sqrt{(c_1+2c_3) \langle x,y\rangle^2+(-2c_3|x|^2+1) |y|^2} $$ is Riemannian.
 
 The second case is $1-2c_3r^2+(c_1+2c_3)s^2= 0$. This implies $c_1=-\frac{1}{r^2}$ and $c_3=\frac{1}{2r^2}$.  In this case, the functions $P$ and $Q$ are given by
 $$P= -\frac{s}{r^2}  , \quad Q=\frac{1}{2}c_0 s^2+\frac{1}{2r^2}.$$
 Hence, the compatibility conditions \eqref{Comp_C_C_2}  reduced to  
 \begin{equation*}
    \begin{split}
        C_2 =& \frac{1}{r}\phi_r-(-\frac{s}{r^2}+c_0 s(r^2-s^2))\phi_s+(\frac{1}{r^2}-c_0s^2) \phi =0.
    \end{split}
\end{equation*}
According to \cite[Lemma 3.3]{Zhou-2}, the general solution of the above equation is given by
$$
\phi=s\ \psi\left(\frac{s^2}{g(r)+s^2 \int 4 r c_0(r) g(r) d r}\right) e^{-\int\left(\frac{2}{r}-2 r^3 c_0(r)\right) d r} 
$$
where $g(r)=e^{\int\left(\frac{2}{r}-4 r^3 c_0(r)\right) d r}$.
\end{proof}

\begin{corollary} \label{corollary}
A spherically symmetric Finsler metric of dimension $n\geq 3$ is Berwaldian if and only if 
$$P=f_1 s  , \quad Q=f_2 s^2+f_3,$$
where $f_1$, $f_2$ and $f_3$ are arbitrary functions of $r$.
\end{corollary} 
\begin{proof}
Assume that $F$ is a spherically symmetric metric with the geodesic spray given by the functions
 $$P=f_1 s  , \quad Q=f_2 s^2+f_3.$$
Then, we have $$P-sP_s=0,\quad P_{ss}=0,\quad Q_s-sQ_{ss}=0, \quad Q_{sss}=0.$$
That is, the Berwald curvature \eqref{Eq:G^i_{jkl}} vanishes and hence the metric is Berwaldian.  Conversely, assume that $F$ is Berwaldian, then by \eqref{Eq:Berwald-P-Q}, the functions $P$ and $Q$ can be written in the form  
$$P=f_1 s  , \quad Q=f_2 s^2+f_3.$$
This completes the proof.
\end{proof}

As a simple class of spherically symmetric Finsler metrics of Berwald type, we have the following class.

\begin{theorem}
All   spherically symmetric metrics in which the function $\phi$ is homogeneous of degree $-1$ in $r$ and $s$ are Berwaldian.   
\end{theorem}
\begin{proof}
Let $F=u\phi$ be a spherically symmetric metric  such that $\phi$ is homogeneous of degree $-1$ in $r$ and $s$. Then by Euler's Theorem of homogeneous functions, we have
$$r\phi_r+s\phi_s=-\phi.$$
Differentiating the above equation with respect to $s$, we get
$$r\phi_{rs}+s\phi_{ss}=-2\phi_s.$$
Then, by substituting from the above two equations into \eqref{P,Q} to get the functions $P$ and $Q$ as follows:  to get $Q$, we substitute by $\phi$ and $\phi_{rs}$. To obtain $P$, we substitute by $\phi$. Therefore, we find that 
$$Q=\frac{1}{2r^2}, \quad P=-\frac{s}{r^2}.$$
By Corollary \ref{corollary}, $F$ is Berwaldian. Moreover, for the above formulae of $P$ and $Q$, one can see  that the compatibility conditions are satisfied. 
\end{proof}

\section{Landsberg surfaces}

 The following theorem completes the classification, done in \cite{Elgendi-SSM}, of Landsberg spherically symmetric Finsler metrics.
 
   \begin{theorem}
    \label{2nd_surface_result}
   All Landsberg spherically symmetric Finsler surfaces are Berwaldain. 
   \end{theorem}
  \begin{proof}
  Let $F=u\phi$ be a Landsberg spherically symmetric Finsler surface.  Then the condition \eqref{Main_Lands_surface_cond}  together with the compatibility conditions   \eqref{Comp_C_C_2} are satisfied. 
  
    Now, at each point $(x,y)\in T\mathbb{B}^n(r_0)\backslash \{0\}$, we can  consider   the   compatibility condition $C_1=0$ and the Landsberg condition \eqref{Main_Lands_surface_cond}    are algebraic homogeneous equations in $\phi$ and $\phi_s$. Therefore, at each point the two equations  must be  dependent so that $\phi$ has  non-zero values,  since in the regular case,  $\phi$ is positive. If we allow the surface to be  singular in certain direction, then it may happen that $\phi$ is zero at that direction, in this case,  the two equations should be dependent  on a subset of $T\mathbb{B}^n(r_0)$.  Otherwise, the functions $\phi$ and $\phi_s$ vanish.   
    
      We claim that the two   equations  are not compatible in the sense that will be explained below and hence the coefficients of $\phi_s$ and $\phi$ in \eqref{Main_Lands_surface_cond} must vanish. Therefore, $sH-(r^2-s^2)H_s=0$  and by Proposition \ref{Berwald_surface}, the surface is Berwaldian.
    
   \medskip 
    
  Now, let's show that the   conditions \eqref{Main_Lands_surface_cond}  and     \eqref{Comp_C_C_2} are not compatible.  We consider  a Berwald surface, then by Theorem \ref{First_surface_result}, the functions $P$ and $Q$ are given by
    \begin{equation}
\label{Eq:Counter}
    \begin{split}
        P&=b_1 s+ \frac{b_2}{ \sqrt{r^2-s^2}} +\frac{b_3 (r^2-2s^2)}{\sqrt{r^2-s^2}} ,\\ 
        Q&= b_0 s^2+\frac{1}{2} b_1+ \frac{b_2 s (r^2-2 s^2)}{r^4 \sqrt{r^2-s^2}} -\frac{b_3 s (3r^2-2s^2)}{r^2\sqrt{r^2-s^2}}-\frac{a  }{r^2} s \sqrt{r^2-s^2}.
    \end{split}
\end{equation}
    Substituting by the above formulae of $P$ and $Q$ into the compatibility condition $C_1$, we get
\begin{equation}
\label{Zhou_phi_2}
    \begin{split}
        \frac{\phi_s}{\phi}= & \frac{2b_1s\sqrt{r^2-s^2}-as^2+2b_3r^2-4b_3s^2+2b_2 }{\sqrt{r^2-s^2}(1-b_1r^2+2b_1s^2+s\sqrt{r^2-s^2}(a+4b_3)) }.  
    \end{split}
\end{equation}

It should be noted that the denominator $1-b_1r^2+2b_1s^2+s\sqrt{r^2-s^2}(a+4b_3)\neq 0$. Indeed, assume that $1-b_1r^2+2b_1s^2+s\sqrt{r^2-s^2}(a+4b_3)= 0$.  Removing the square root and combining the like terms imply
$$\left( {a}^{2}+8 a{ b_3}+4{{ b_1}}^{2}+16 {{ b_3}}^{2}
 \right) {s}^{4}- \left( {a}^{2}{r}^{2}+8 a{ b_3} {r}^{2}+4 {{
 b_1}}^{2}{r}^{2}+16 {{ b_3}}^{2}{r}^{2}-4 { b_1} \right) {s}^
{2}+{{ b_1}}^{2}{r}^{4}-2 { b_1} {r}^{2}+1
=0.$$
Since the above equation is valid for all $s$, we have $b_1=\frac{1}{r^2}$, $a=-4b_3$. Substituting these formulae in to the coefficient of $s^4$,   we have
$${a}^{2}+8 a{ b_3}+4{{ b_1}}^{2}+16 {{ b_3}}^{2}=4\neq0.$$
Which is a contradiction. 

\medskip

Now, we can rewrite the Landsberg condition \eqref{Main_Lands_surface_cond} as follows
    \begin{equation}
    \begin{split}
        &(s\phi+(r^2-s^2)\phi_s)((r^2-s^2)Q_{sss}+3(Q_s-sQ_{ss}))\\
       &+((r^2-s^2)P_{sss}-3sP_{ss})\phi+(3(r^2-s^2)P_{ss}+3(P-sP_s))\phi_s=0.   
    \end{split}
\end{equation}
By the formula of $\frac{\phi_s}{\phi}$ given in \eqref{Zhou_phi_2}, one can see that the compatibility condition $C_1$ is satisfied. 

For a choice of the functions $b_2(r)$ and $b_3(r)$ such that $b_2-b_3r^2\neq 0$, we have
$$(r^2-s^2)Q_{sss}+3(Q_s-sQ_{ss})=-\frac{6r^2(b_2-b_3r^2)}{(r^2-s^2)^{5/2}}\neq 0.
$$
 Then, we have
\begin{equation}
\label{Eq:L-1}
    \begin{split}
      (s\phi+(r^2-s^2)\phi_s)  =-\frac{((r^2-s^2)P_{sss}-3sP_{ss})\phi+(3(r^2-s^2)P_{ss}+3(P-sP_s))\phi_s}{(r^2-s^2)Q_{sss}+3(Q_s-sQ_{ss})}
    \end{split}
\end{equation}
Moreover, the condition $C_1$ can be rewritten in the form
\begin{equation}
\label{Eq:C1-1}
       C_1 = -(s \phi+(r^2-s^2)\phi_s)(2Q-sQ_s)+(1+sP)\phi_s-(sP_s-2P)\phi =0
\end{equation}
Now, substituting from \eqref{Eq:L-1} into \eqref{Eq:C1-1}, we get
    \begin{equation}
    \label{Eq:LC-1}
    \begin{split}
         & \frac{(2Q-sQ_s)[((r^2-s^2)P_{sss}-3sP_{ss})\phi+(3(r^2-s^2)P_{ss}+3(P-sP_s))\phi_s]}{(r^2-s^2)Q_{sss}+3(Q_s-sQ_{ss})}\\
          &+(1+sP)\phi_s+(sP_s-2P)\phi =0
        \end{split}
\end{equation}
 Dividing both sides of \eqref{Eq:LC-1} by $\phi$ we have
\begin{equation}
\label{Eq:LC-2}
    \begin{split}
         & \frac{(2Q-sQ_s)[((r^2-s^2)P_{sss}-3sP_{ss})+(3(r^2-s^2)P_{ss}+3(P-sP_s))\frac{\phi_s}{\phi}]}{(r^2-s^2)Q_{sss}+3(Q_s-sQ_{ss})}\\
          &+(1+sP)\frac{\phi_s}{\phi}+(sP_s-2P)  =0
        \end{split}
\end{equation}
The substitution from \eqref{Eq:Counter} and \eqref{Zhou_phi_2} into \eqref{Eq:LC-2} implies that the left hand side of \eqref{Eq:LC-2} is non zero. For example, consider the choice 
$$a=0,\quad b_1=-\frac{1}{r^2}, \quad b_2=r^2, \quad b_3=0.$$
Then,  the left hand side of \eqref{Eq:LC-2} is given by
$$\frac{s((r^2-s^2)^2(1+r^2)+r^8)-r^6\sqrt{r^2-s^2}(r^2-s^2+1)}{r^2(r^2-s^2)^2}.$$
And this is a contradiction and the proof is completed.
   \end{proof}
   
   \begin{remark}
   It should be noted that the above proof of the above theorem dose not depend on whether the surface is regular or not.
   \end{remark}
   
\begin{landscape}
\begin{center}{\bf{Table 1: Classification of Landsberg and Berwald spherically symmetric metrics}}
\end{center}
{\begin{table}[h]
\centering
\begin{tabular}{|c|cc|cc|}
\hline
\multirow{2}{*}{\textbf{Type}}      & \multicolumn{2}{c|}{\multirow{2}{*}{\textbf{Landsberg}}}                                                          & \multicolumn{2}{c|}{\multirow{2}{*}{\textbf{Berwald}}}                                       \\
                           & \multicolumn{2}{c|}{}           & \multicolumn{2}{c|}{}                                                    \\ \hline
\multirow{2}{*}{\textbf{Dimension} }  & \multicolumn{1}{c|}{\multirow{2}{*}{$n=2$}} & \multirow{2}{*}{$n\geq 3$} & \multicolumn{1}{c|}{\multirow{2}{*}{$n=2$}} & \multirow{2}{*}{$n\geq 3$} \\ 
 & \multicolumn{1}{c|}{}                  & \multicolumn{1}{c|}{}                       & \multicolumn{1}{c|}{}                  &                        \\  \hline
\multirow{2}{*}{P}           & \multicolumn{1}{c|}{\multirow{2}{*}{as the Berwaldian surface}}    &    \multirow{2}{*}{ $P=c_1 s+\frac{c_2}{r^2}\sqrt{r^2-s^2}$    }               & \multicolumn{1}{c|}{\multirow{2}{*}{$P=b_1 s+ \frac{b_2}{ \sqrt{r^2-s^2}} +\frac{b_3 (r^2-2s^2)}{\sqrt{r^2-s^2}} $}  }  &     \multirow{2}{*}{        $ P=c_1 s $     }   \\ 
  & \multicolumn{1}{c|}{}                  & \multicolumn{1}{c|}{}                       & \multicolumn{1}{c|}{}                  &                        \\ \hline
\multirow{2}{*}{Q}           & \multicolumn{1}{c|}{\multirow{2}{*}{as the Berwaldian surface}}    &      \multirow{2}{*}{ $ Q=\frac{1}{2}c_0 s^2-\frac{c_2 s}{r^4}\sqrt{r^2-s^2}  $    }       & \multicolumn{1}{c|}{\multirow{2}{*}{$Q=b_0 s^2+\frac{1}{2} b_1+ \frac{b_2 s (r^2-2 s^2)}{r^4 \sqrt{r^2-s^2}} $}}    &    \multirow{2}{*}{         $Q=\frac{1}{2}c_0 s^2+c_3$ }    \\ 
 & \multicolumn{1}{c|}{}                  & \multicolumn{1}{c|}{}                       & \multicolumn{1}{c|}{}                  &                        \\ 
 & \multicolumn{1}{c|}{}                  & \multicolumn{1}{c|}{$+c_3$}                       & \multicolumn{1}{c|}{$-\frac{b_3 s (3r^2-2s^2)}{r^2\sqrt{r^2-s^2}}-\frac{a  }{r^2} s \sqrt{r^2-s^2}$}                  &                        \\  
  & \multicolumn{1}{c|}{}                  & \multicolumn{1}{c|}{}                       & \multicolumn{1}{c|}{}                  &                        \\ \hline
\multirow{2}{*}{\textbf{Regular examples} }    & \multicolumn{4}{c|}{  \multirow{2}{*}{   only Berwaldian examples exist  }          }            \\
                
                                                                                             & \multicolumn{4}{c|}{}                                                               \\ 
 \hline
\multirow{2}{*}{\textbf{Non-regular examples}}& \multicolumn{1}{c|}{\multirow{2}{*}{Berwaldian surfaces exist }  }    &  \multirow{2}{*}{   Some examples exist  }                 & \multicolumn{1}{c|}{\multirow{2}{*}{Some examples exist}}    &     \multicolumn{1}{c|}{\multirow{2}{*}{Some examples exist}}              \\ 
 & \multicolumn{1}{c|}{}                  & \multicolumn{1}{c|}{}                       & \multicolumn{1}{c|}{}                  &                        \\ \hline
\multirow{2}{*}{\textbf{Concrete examples}}    & \multicolumn{1}{c|}{\multirow{2}{*}{$\displaystyle{F=\frac{u}{r}\phi\left(\frac{s}{r}\right)}$}}    &                       \multicolumn{1}{c|}{\multirow{2}{*}{see \cite[Examples 1 and 2]{Elgendi-SSM} }}  & \multicolumn{1}{c|}{\multirow{2}{*}{$\displaystyle{F=\frac{u}{r}\phi\left(\frac{s}{r}\right)}$}}    &      \multicolumn{1}{c|}{\multirow{2}{*}{$\displaystyle{F=\frac{u}{r}\phi\left(\frac{s}{r}\right)}$}}                     \\ 
 & \multicolumn{1}{c|}{}                  & \multicolumn{1}{c|}{}                       & \multicolumn{1}{c|}{}                  &                        \\ \hline
\end{tabular}
\end{table}}
Where, $c_0$, $c_1$, $c_2$, $c_3$, $b_1$, $b_2$, $b_3$, and $a$ are arbitrary functions of $r$. 
\end{landscape}


\providecommand{\bysame}{\leavevmode\hbox
to3em{\hrulefill}\thinspace}
\providecommand{\MR}{\relax\ifhmode\unskip\space\fi MR }
\providecommand{\MRhref}[2]{%
  \href{http://www.ams.org/mathscinet-getitem?mr=#1}{#2}
} \providecommand{\href}[2]{#2}

\end{document}